\documentclass[submission,copyright,creativecommons]{eptcs}
\providecommand{\event}{ACT 2023} 
\usepackage{breakurl}             
\usepackage{underscore}           
\usepackage[english]{babel}	

\usepackage{amssymb,amsmath,amsthm}       %
\usepackage{amsfonts}      %

\newtheorem{theorem}{Theorem}
\newtheorem{lemma}{Lemma}
\newtheorem{proposition}{Proposition}
\newtheorem{corollary}{Corollary}
\newtheorem{example}{Example}

\usepackage{mathtools}     %

\usepackage{verbatimbox}

\usepackage{tikz}
\usepackage[all]{xy}
\usetikzlibrary{arrows,shapes}
\usetikzlibrary{decorations.pathmorphing}
\usetikzlibrary{decorations.markings}

\title{Magnitude of arithmetic scalar and matrix categories}
\author{Steve Huntsman
\institute{STR}
\email{steve.huntsman@str.us}
}
\def\titlerunning{Magnitude of arithmetic scalar and matrix categories}
\def\authorrunning{S. Huntsman}
\begin{document}
\maketitle

\begin{abstract}
We develop tools for explicitly constructing categories enriched over generating data and that compose via ordinary scalar and matrix arithmetic arithmetic operations. We characterize meaningful size maps, weightings, and magnitude that reveal features analogous to outliers that these same notions have previously been shown to reveal in the context of metric spaces. Throughout, we provide examples of such ``outlier detection'' relevant to the analysis of computer programs, neural networks, cyber-physical systems, and networks of communications channels.
\end{abstract}

\section{\label{sec:intro}Introduction}

We consider a fairly general model $\mathbf{C}$ of a finite state machine or similar structure in which matrix data associated to states and/or transitions compose coherently, compatibly, and scalably with ordinary scalar and matrix arithmetic (e.g., Jacobians in numerical programs; filters in signal processing systems, etc.). That is, $\mathbf{C}$ is a category enriched \cite{kelly1982basic,fong2019invitation} over a suitable category $\mathbf{M}$ of matrices, with the underlying scalars as a special (and as it turns out, universal) case. For a scalar example, see Figure \ref{fig:scalarExample}. 

As a motivating conceptual example that is mathematically archetypal, consider a system of networked components with various external inputs and outputs. It might be that some co-located inputs must take different internal paths through the system because of engineering considerations, yet still necessary to instantiate behavior that is completely independent of the path taken. In the linear time-invariant setting \cite{antsaklis1997linear,liberzon2003switching,spivak2015steady,bakirtzis2021categorical}, this essentially means multiplying different matrices along different paths, yet producing the same result at any point where paths meet. 
\footnote{
In situations where there might be path dependence, a workaround is to consider finite sub-polytrees of the universal cover of $D$ \cite{huntsman2022magnitude}. This is akin to loop unrolling for a compiler \cite{cooper2011engineering}. A violation of path-independent compositionality would indicate nonzero curvature (that physicists call a ``field strength tensor'') in a principal bundle over a discrete directed base \cite{dimakis1994differential,dimakis1994discrete,maldacena2015symmetry}. 
}
This is the case of taking matrix multiplication (or perhaps componentwise multiplication in a Fourier basis) as monoidal product \cite{mac2013categories,fong2019invitation} in $\mathbf{M}$.

We aim to probe the geometry of $\mathbf{C}$ by analogy with subsets of Euclidean space, where the \emph{weightings} recalled in \S \ref{sec:magnitude} provide an excellent scale-dependent boundary/outlier detection mechanism \cite{willerton2009heuristic,bunch2020practical,huntsman2022diversity}. This boundary-detecting behavior is related to the potential-theoretical notion of Bessel capacities \cite{meckes2015magnitude}. 

The following paragraphs give a bit more specificity. For a finite digraph $D = (V,A)$, the \emph{category $\langle D \rangle$ determined by $D$} is a digraph with objects/vertices $V(D)$ and morphisms/arcs $A(\langle D \rangle)$ given by self-loops on all vertices along with any arcs in the transitive closure of $D$. We want to use ``generating'' data in such a way as to obtain a $\mathbf{M}$-category $\mathbf{C}$ with underlying category $\langle D \rangle$. Such a construction can model memoryless systems in which series of transitions produce effects that only depend on the initial and final states. A similar construction in which nondeterministic automata are endowed with a scalar ``cost function'' on inputs was briefly considered in \S 5.3 of \cite{cho2019quantales}.

Taking $\mathbf{M}$ to be the discrete monoidal category $M_n(R)$ of $n \times n$ matrices over a commutative ring $R$ with product given by matrix multiplication informs switched linear systems \cite{liberzon2003switching}. 
Taking $\mathbf{M}$ to be $\text{Arr}(\mathbf{FinVect})$ or $\mathbf{FinStoch}$ with product given by the ordinary (Kronecker) tensor product informs quantum circuits \cite{nielsen2010quantum} and the information theory of networks of discrete memoryless channels \cite{shannon1957certain}. 
We treat these examples in turn, illustrating how to find ``outliers'' in structures directly relevant to computer program analysis, neural networks, cyber-physical systems, and networks of communications channels.

We can do this economically because the semiring $S$ involved in the constructions below is universal in the following sense: there exists a unique $S$-category $\mathbf{C}$ such that $Z_{jk} = \mathbf{C}(j,k)$, and subsequent calculations only rely on ordinary matrix arithmetic.
\footnote{
Recall that a semiring $S$ is a discrete monoidal category with auxiliary additive structure.
}
In the cases we will deal with $\mathbf{M}$ is parameterized 
by a dimension, with the dimension one case essentially recovering $S$, so that the \emph{size map} introduced in \S \ref{sec:magnitude} is essentially a projection. In this sense, considering the \emph{magnitude} of $\mathbf{M}$-categories \emph{versus} $S$-categories \emph{per se} is mostly uninteresting apart from the identification of a size map, and it factors through a universal case. However, the \emph{construction} of $\mathbf{M}$-categories introduces a bit more nuance than the construction of $S$-categories and also gives important context for applications.

The paper is organized as follows. \S \ref{sec:magnitude} reviews the category-theoretic formulation of magnitude. \S \ref{sec:Scalar} discusses the existence, construction, and analysis of categories enriched over ordinary scalar arithmetic, with examples relevant to the analysis of computer programs and structures informing neural networks. \S \ref{sec:Matrix} extends these results to the context of categories enriched over ordinary matrix arithmetic, with an example relevant to networks of communications channels. Finally, \S \ref{sec:Remarks} makes some speculative remarks on combining arithmetic and process-oriented data.

\section{\label{sec:magnitude}Magnitude}

Let $\mathbf{M} = (\mathbf{M},\otimes,1)$ be a monoidal category \cite{mac2013categories,fong2019invitation} and $\mathbf{C}$ a finite $\mathbf{M}$-category. Recall that this means that $\mathbf{C}$ is specified by a finite set $\text{Ob}(\mathbf{C})$; hom-objects $\mathbf{C}(j,k) \in \mathbf{M}$ for all $j, k \in \text{Ob}(\mathbf{C})$; identity morphisms $1 \rightarrow \mathbf{C}(j,j)$ for all $j \in \text{Ob}(\mathbf{C})$; and composition morphisms $\mathbf{C}(j,k) \otimes \mathbf{C}(k,\ell) \rightarrow \mathbf{C}(j,\ell)$ for all $j, k, \ell \in \text{Ob}(\mathbf{C})$, all satisfying associativity and unitality properties \cite{kelly1982basic,fong2019invitation}.

The theory of magnitude \cite{leinster2017magnitude,leinster2021entropy} takes two principal inputs. The first input is a $\mathbf{M}$-category $\mathbf{C}$. The second input is a \emph{size map} $\sigma : \text{Ob}(\mathbf{M}) \rightarrow S$ where $S$ is a semiring. 
\footnote{
In our context, when $\mathbf{M}$ is a semiring like $M_n(R)$ it is possible--and may be useful--to take $\sigma$ to be the identity so $S = \mathbf{M}$.
}
The size map is required to be constant on isomorphism classes and to satisfy $\sigma(1) = 1$ and $\sigma(X \otimes Y) = \sigma(X) \odot \sigma(Y)$, where the semiring unit and multiplication are indicated on the right-hand sides. If $n := |\text{Ob}(\mathbf{C})| < \infty$ then its \emph{similarity matrix} $Z \in M_n(S)$ is given by $Z_{jk} := \sigma(\mathbf{C}(j,k))$. Introducing the (common) notation $$(f[X])_{jk} := f(X_{jk})$$ as a shorthand where $X$ is a matrix over $S$ and $f$ is a function on $S$, we have $Z := \sigma[\mathbf{C}]$.

A \emph{weighting} is a column vector $w$ satisfying $Zw = 1$, where the semiring matrix multiplication and column vector of ones are indicated. A \emph{coweighting} is the transpose of a weighting for $Z^T$. If $Z$ has both a weighting and a coweighting, its \emph{magnitude} is the sum of the components (both sums coincide).

Presently, examples and applications of magnitude are focused almost entirely on \emph{Lawvere metric spaces}, i.e., categories enriched over $([0,\infty],\ge)$ where the monoidal product is ordinary addition. These are also known as \emph{extended quasipseudometric spaces} since they generalize metric spaces by allowing distances that are infinite (extended), asymmetric (quasi-), or zero (pseudo-). As far as we are aware, the only exceptions to this focus at present are this paper and \cite{chuang2016magnitude,huntsman2022magnitude}.

\section{\label{sec:Scalar}Scalar categories}

Let $S \in \{\mathbb{R},\mathbb{C}\}$ and consider the system 
\begin{equation}
\label{eq:multiplicative}
Z_{jk} Z_{k\ell} = Z_{j\ell}
\end{equation}
for $(j,k), (k,\ell), (j, \ell) \in A(\langle D \rangle)$.
\footnote{
If we write $d := \tau \log[Z]$ for a suitable scalar $\tau$, then \eqref{eq:multiplicative} takes the form of the triangle equality $d_{jk} + d_{k\ell} = d_{j\ell}$, which highlights a similarity with ``vanilla'' magnitude of Lawvere metric spaces that obey a criterion similar to Menger convexity \cite{leinster2021magnitude}. There are other related notions: a weighted undirected graph is called \emph{graph-geodetic} (respectively, \emph{cutpoint-additive}) when the preceding triangle inequality holds if (respectively, iff) every path from $j$ to $\ell$ passes through $k$ \cite{klein1998distances,deza2009encyclopedia,chebotarev2020hitting}. 
}
This system is satisfied iff $Z$ defines a $S$-category $\mathbf{C}$ such that $Z_{jk} = \mathbf{C}(j,k)$. A class of solutions to \eqref{eq:multiplicative} is 
\begin{equation}
\label{eq:vertexData}
Z_{jk} := p_j^{-1} p_k
\end{equation}
for $p : V(D) \rightarrow S_\times$. We will show that \eqref{eq:vertexData} turns out to be the general nondegenerate case. 
For reasons that will become apparent in \S \ref{sec:Matrix1}, it will be helpful to establish when integral solutions exist. 

Let $\tilde A(D)$ denote the arcs in $D$ that are not loops. Let $\Gamma_2(D) := \{(j,k,\ell) : (j,k), (k,\ell) \in \tilde A(\langle D \rangle); j \ne \ell \}$ be the set of nondegenerate paths in $\langle D \rangle$ of length two. If $|\Gamma_2(D)| = 0$, then we can trivially obtain all solutions of \eqref{eq:multiplicative}, so assume w.l.o.g. $|\Gamma_2(D)| > 0$. Lemma \ref{lemma:multiplication} gives a general solution to \eqref{eq:multiplicative} (see Figure \ref{fig:scalarExample}).

\begin{lemma}
\label{lemma:multiplication}
Define a $|\Gamma_2(D)| \times |\tilde A(\langle D \rangle)|$ matrix $M$ to have entries that are zero except for $M_{(j,k,\ell),(j,k)} = 1$, $M_{(j,k,\ell),(k,\ell)} = 1$, and $M_{(j,k,\ell),(j,\ell)} = -1$. Then $m:= \dim \ker M > 0$ and any solution to \eqref{eq:multiplicative} is of the form
\begin{equation}
\label{eq:kerAssign}
Z_{jk} = \prod_{i=1}^m c_i^{y^{(i)}_{(j,k)}}
\end{equation}
where $\{y^{(i)}\}_{i = 1}^m$ is a basis for $\ker M$ and $(c_i)_{i=1}^m \in S^m$.
\end{lemma}

\begin{proof}
W.l.o.g., assume that $\langle D \rangle$ is weakly connected. If $(j,k,\ell) \in \Gamma_2$, then by transitivity of composition $(j,\ell) \in \tilde A(\langle D \rangle)$ and similarly $|\Gamma_2(D)| < |\tilde A(\langle D \rangle)|$. $M y = 0$ iff $y_{(j,k)} + y_{(k,\ell)} - y_{(j,\ell)} = 0$, so we can take $y_{(j,k)} = \log Z_{jk}$. Since $|\Gamma_2(D)| < |\tilde A(\langle D \rangle)|$, we have $m:= \dim \ker M > 0$.
\end{proof}

By normalizing the reduced row echelon form, we get $y^{(i)} \in \mathbb{Z}^{| \tilde A(\langle D \rangle)|}$. If $(c_i)_{i=1}^m \in \mathbb{Z}_+^m$, then $Z_{jk} \in \mathbb{Q}_+$.

\begin{proposition}
\label{proposition:multiplication}
If $D$ (or equivalently, $\langle D \rangle$) is a directed acyclic graph (DAG), the system \eqref{eq:multiplicative} admits nontrivial solutions over $\mathbb{Z}_+$. More generally, $\prod_{(j,k) \in \gamma} Z_{jk} = 1$ for any cycle $\gamma \in \langle D \rangle$, so any solution over $\mathbb{Z}_+$ must be unity on any arcs that are in a cycle: in particular, if $(j,k)$ is in a cycle, then $Z_{jk} = Z_{kj}^{-1}$.
\end{proposition}

We now aim at closed forms for generating data on arcs and on vertices. Let $U$ be the functor from quivers to undirected graphs that forgets arc orientations and multiplicities. For $(i,i') \in V(D)^2$, write $$\varepsilon(i,i') := \begin{cases} 1 & \text{ if } (i,i') \in A(D) \\ -1 & \text{ if } (i',i) \in A(D) \text{ and } (i,i') \notin A(D) \\ 0 & \text{ otherwise } \end{cases}.$$

\begin{theorem}
\label{theorem:multiplication}
Let $D$ be weak (= weakly connected), $T$ be a spanning tree of $U(D)$, and $W : E(T) \rightarrow S$ nondegenerate, where here as usual $E(\cdot)$ indicates the edges of an undirected graph. The assignment
\begin{equation}
\label{eq:treeAssign}
Z_{jk} = \prod_{(i,i') \in T[j,k]} W(i,i')^{\varepsilon(i,i')},
\end{equation}
with the product over edges in the path $T[j,k]$ in $T$ from $j$ to $k$, is well-defined and satisfies \eqref{eq:multiplicative}. 
If $D$ is a DAG, then $W : E(T) \rightarrow \mathbb{Z}_+$ yields $Z_{jk} \in \mathbb{Z}_+$. Finally, any nondegenerate solution of \eqref{eq:multiplicative} has the form \eqref{eq:treeAssign}.
\end{theorem}

\begin{proof}
The assignment \eqref{eq:treeAssign} is well-defined since $T[j,k]$ exists; it is unique since $D$ is weak. The assignment satisfies \eqref{eq:multiplicative} because the concatenation of $T[j,k]$ and $T[k,\ell]$ is $T[j,\ell]$. If $D$ is a DAG, then $\varepsilon \equiv 1$. Finally, by transitivity, any nondegenerate solution of \eqref{eq:multiplicative} has the form \eqref{eq:treeAssign}.
\end{proof}

\begin{theorem}
\label{theorem:characterization}
If $D$ is weak, any nondegenerate solution of \eqref{eq:multiplicative} has the form $Z_{jk} = p_j^{-1} p_k$ for some $p$. 
\footnote{
If $B$ is the incidence matrix of $D$, then $\log [\text{vec } Z] = B^T \log [p]$, where $\text{vec}$ stacks matrix columns and $p$ is taken as a vector.
} 
\end{theorem}

\begin{proof}
Let $Z$ satisfy \eqref{eq:multiplicative} and let $P$ be a spanning polytree of $D$, i.e., a spanning digraph such that $T := U(P)$ is a tree. Pick $i \in V(D)$ and set $p_i = 1$, then extend $p$ to $V(P) = V(D)$ by traversing $P$ and solving $Z_{jk} = p_j^{-1} p_k$ on $A(P)$. Finally, take $W(j,k) := Z_{jk}^{\varepsilon(j,k)}$ on $T$ and apply Theorem \ref{theorem:multiplication} to recover $Z$.
\end{proof}

\begin{example}
\label{example:scalar}
Figure \ref{fig:scalarExample} shows an example that is essentially generic in light of the structural characterization of transitive digraphs in Proposition 2.3.1 of \cite{bang2008digraphs}. However, its magnitude is undefined.
\begin{figure}[htbp]
  \centering
  \includegraphics[trim = 50mm 90mm 40mm 85mm, clip, width=.49\textwidth,keepaspectratio]{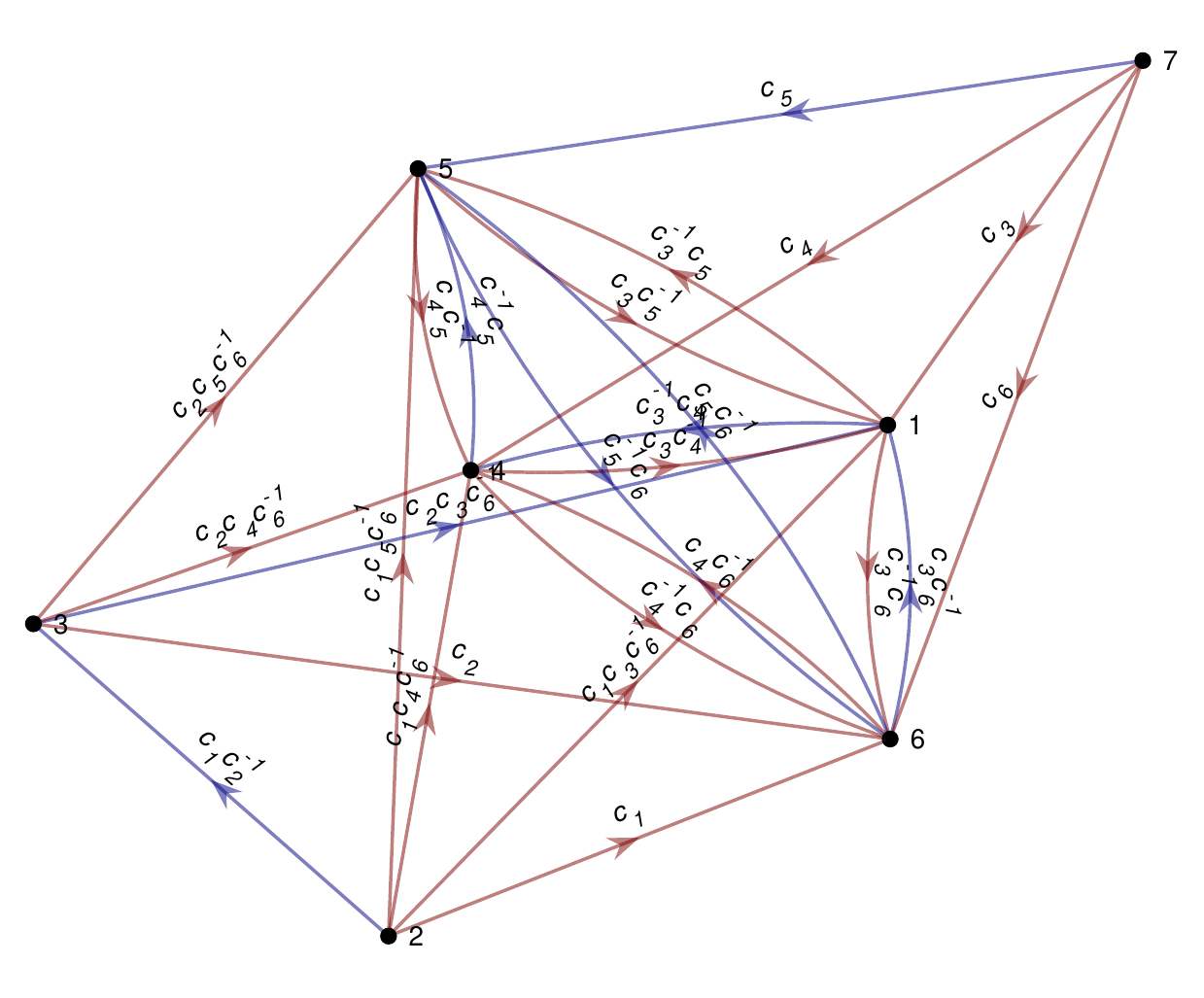} 
  \includegraphics[trim = 50mm 90mm 40mm 85mm, clip, width=.49\textwidth,keepaspectratio]{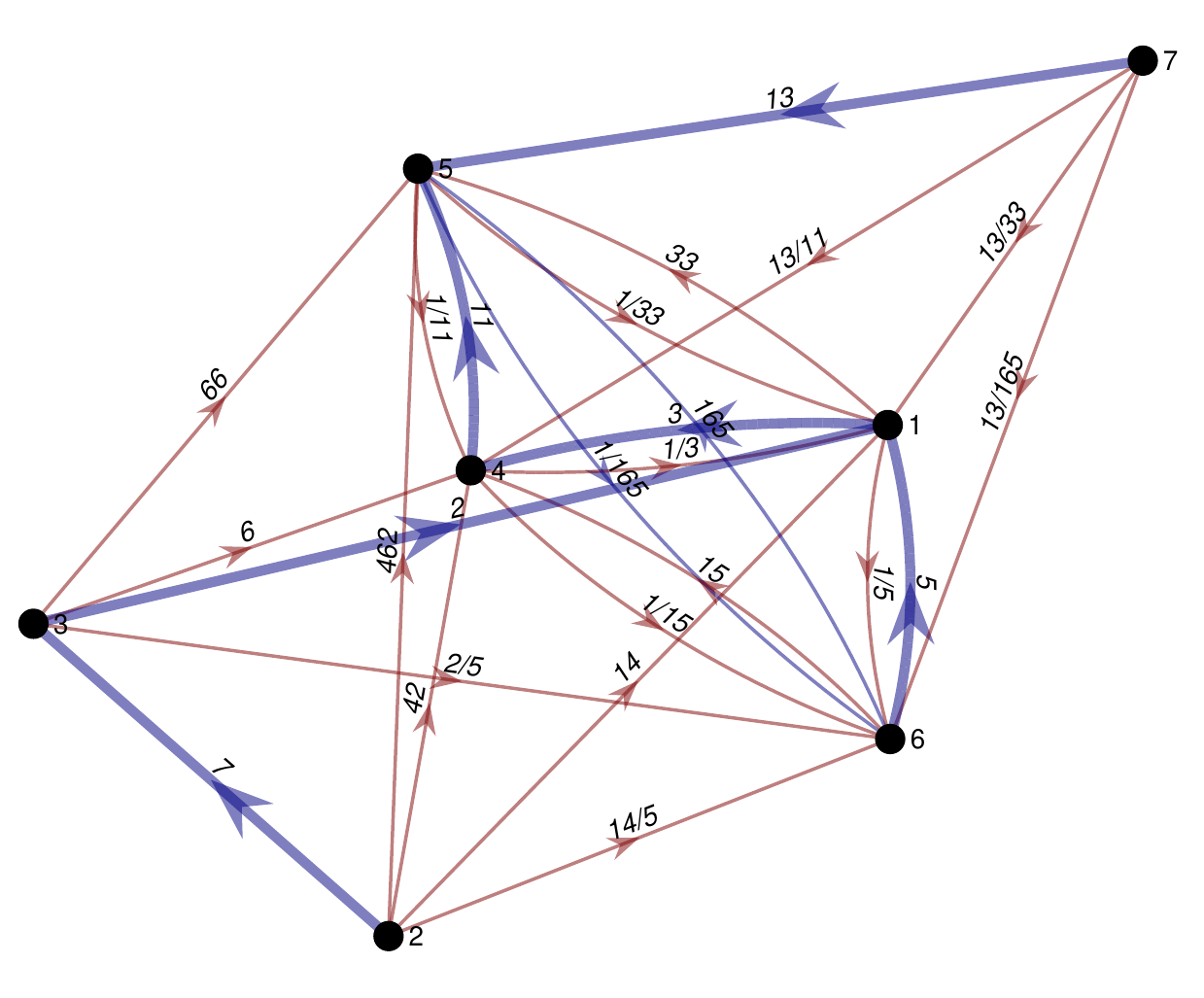} 
  \caption{\label{fig:scalarExample} (Left) A symbolic solution to \eqref{eq:multiplicative} of the form \eqref{eq:kerAssign} on {\color{red}a transitive digraph (red)} generated from {\color{blue}an underlying digraph (blue)}. (Right) A specific solution to \eqref{eq:multiplicative} generated using the assignment \eqref{eq:treeAssign} for the {\color{blue}{\bf highlighted spanning (poly)tree}}. Taking $p \propto (14,1,7,42,462,14/5,462/13)$ yields $Z_{jk} = p_j^{-1}p_k$.}
\end{figure}
\end{example}

The simplest solutions to \eqref{eq:multiplicative} fail to give interesting structure.

\begin{lemma} 
\label{lemma:ordinaryExistence} 
If $Z_{jk} = Z_{kj}^{-1}$ for all $(j,k)$, then $\mathbf{C}$ does not have well-defined magnitude unless $Z_{jk} \equiv 1$.
\end{lemma}

\begin{proof}
By hypothesis $Z_{j\ell} = Z_{jk} Z_{k\ell} = Z_{kj}^{-1} Z_{k\ell}$ for all $\ell$. Since the $j$th and $k$th rows of $Z$ are constant multiples of each other, the equation $Zw = 1$ only has a solution if $Z_{jk} = 1$.
\end{proof}

Similar considerations also show that if $D$ has a cycle (that is not a loop) and $\mathbf{C}$ has magnitude, the magnitude must be unity. However, the space of weightings still encodes nontrivial information about $\mathbf{C}$.

\begin{example}
Consider a toy program that is constructed as follows. We generate a program ``skeleton'' using productions from the probabilistic context free grammar \cite{visnevski2007syntactic}
$$\texttt{S} \rightarrow \texttt{S; S} \ | \ \texttt{if b; S; fi} \ | \ \texttt{while b; S; end}$$
where $\texttt{S}$ is shorthand for a line separator: the production probabilities are respectively $0.6$, $0.1$, and $0.3$. The tokens $\texttt{S}$ and $\texttt{b}$ respectively represent statements/subroutines and Boolean predicates. 

Next, we form the resulting \emph{control flow graph} \cite{cooper2011engineering} by associating vertices with lines in the skeleton and edges according to Table \ref{tab:CFG}. We also prepend a \texttt{START} line/vertex/arc and append a \texttt{HALT} line/vertex/arc to both the program and the control flow graph.
\begin{table}
  \centering
  \caption{Control flow graph arc: $[\cdot] := $ line number of matching token.}
  \label{tab:CFG}
  \begin{tabular}{ccl}
    source at line $j$& target($\top$)&target($\bot$)\\
    \hline
    \texttt{if b} & $j+1$& [\texttt{fi}]+1 \\
    \texttt{while b} & $j+1$& [\texttt{end}]+1 \\
    \texttt{end} & [\texttt{while}]& $\cdot$ \\
    \texttt{fi} or \texttt{S} & $j+1$& $\cdot$ \\
\end{tabular}
\end{table}
We can explicitly instantiate an executable program by i) replacing a token \texttt{S} on line $j$ of the program by an explicit statement 
and ii) replacing a token \texttt{b} on line $k$ of the program by an explicit predicate. 

Consider the following assignment of scalar data to arcs of the control flow graph:
\begin{itemize}
	\item An arc of the form $(\texttt{S}_j,\cdot_{j+1})$ is assigned $\delta_j \in S$;
	\item All other arcs not of the form $(\texttt{if},\cdot_{[\texttt{fi}]+1})$ or $(\texttt{end},\texttt{while}_{[\texttt{while}]})$ are assigned $1 \in S$;
	\item An arc of the form $(\texttt{if},\cdot_{[\texttt{fi}]+1})$ is assigned the (ordered) product of data assigned to any other path between its source and target;
	\item An arc of the form $(\texttt{end},\texttt{while}_{[\texttt{while}]})$ is assigned the inverse of the (ordered) product of data assigned to any other path between its target and source.

\end{itemize}

\begin{proposition}
\label{proposition:partialAssignment} 
The assignment above is well defined and uniquely corresponds to an $S$-category. \qed
\end{proposition}

Consider the control flow graph obtained with $20$ context free grammar productions shown in Figure \ref{fig:ControlFlowGraph} and the assignment $\delta_j = 2$ for all $j$. The resulting similarity matrix and its kernel are respectively shown in the left and right panels of Figure \ref{fig:Matrices}. The space of weightings is obtained by adding the vector $(1,0,\dots,0)^T$ to the kernel: this vector corresponds to the $\texttt{START}$ vertex. The space of coweightings is similar (not shown). The basis vectors in the kernel all sum to zero, so the magnitude is always unity.

\begin{figure}[htbp]
  \centering
  \includegraphics[trim = 50mm 90mm 30mm 90mm, clip, width=.5\textwidth,keepaspectratio]{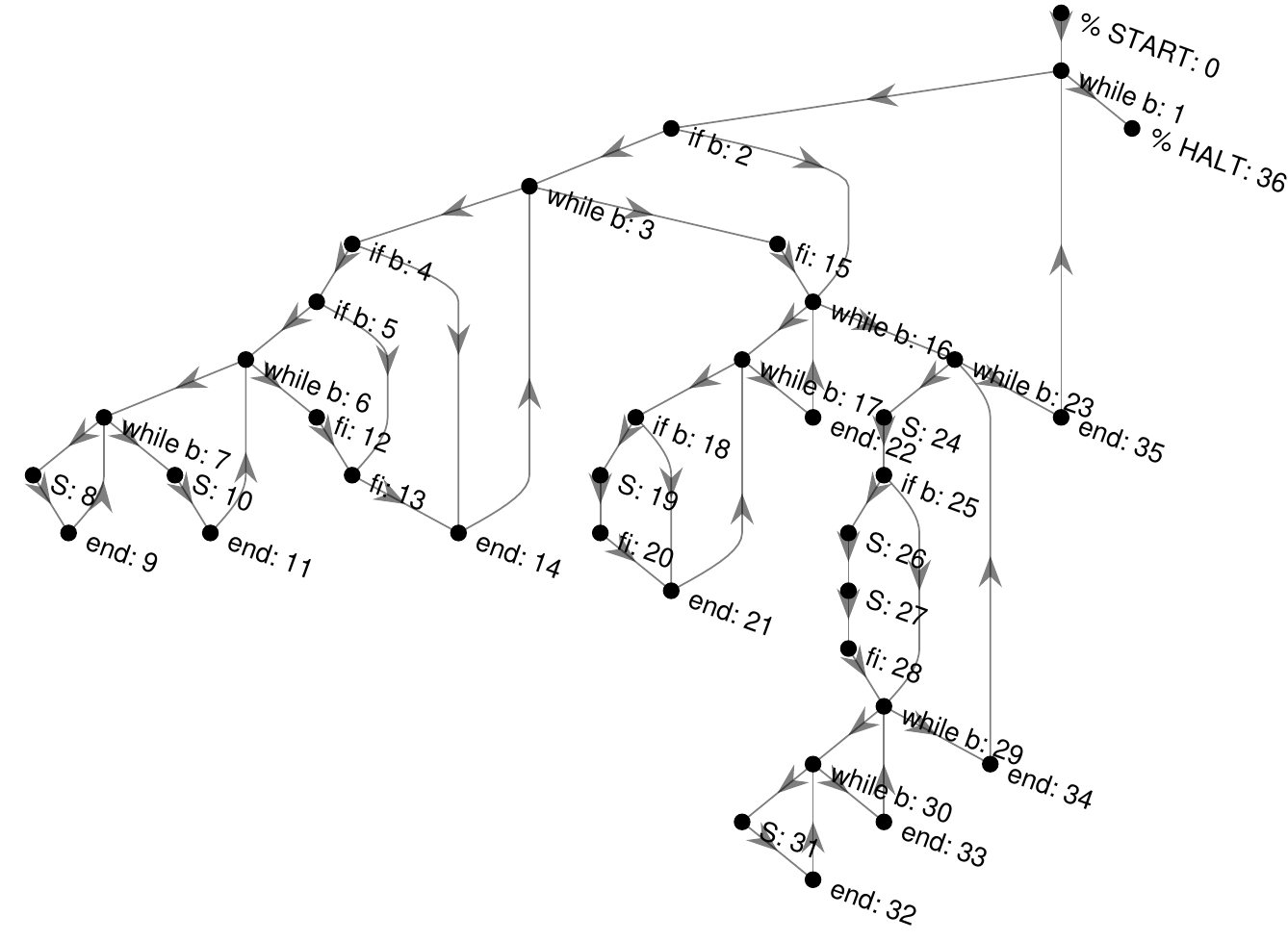} 
  \caption{\label{fig:ControlFlowGraph} A simple control flow graph.}
\end{figure}

\begin{figure}[htbp]
  \centering
  \includegraphics[trim = 20mm 65mm 20mm 70mm, clip, width=.49\textwidth,keepaspectratio]{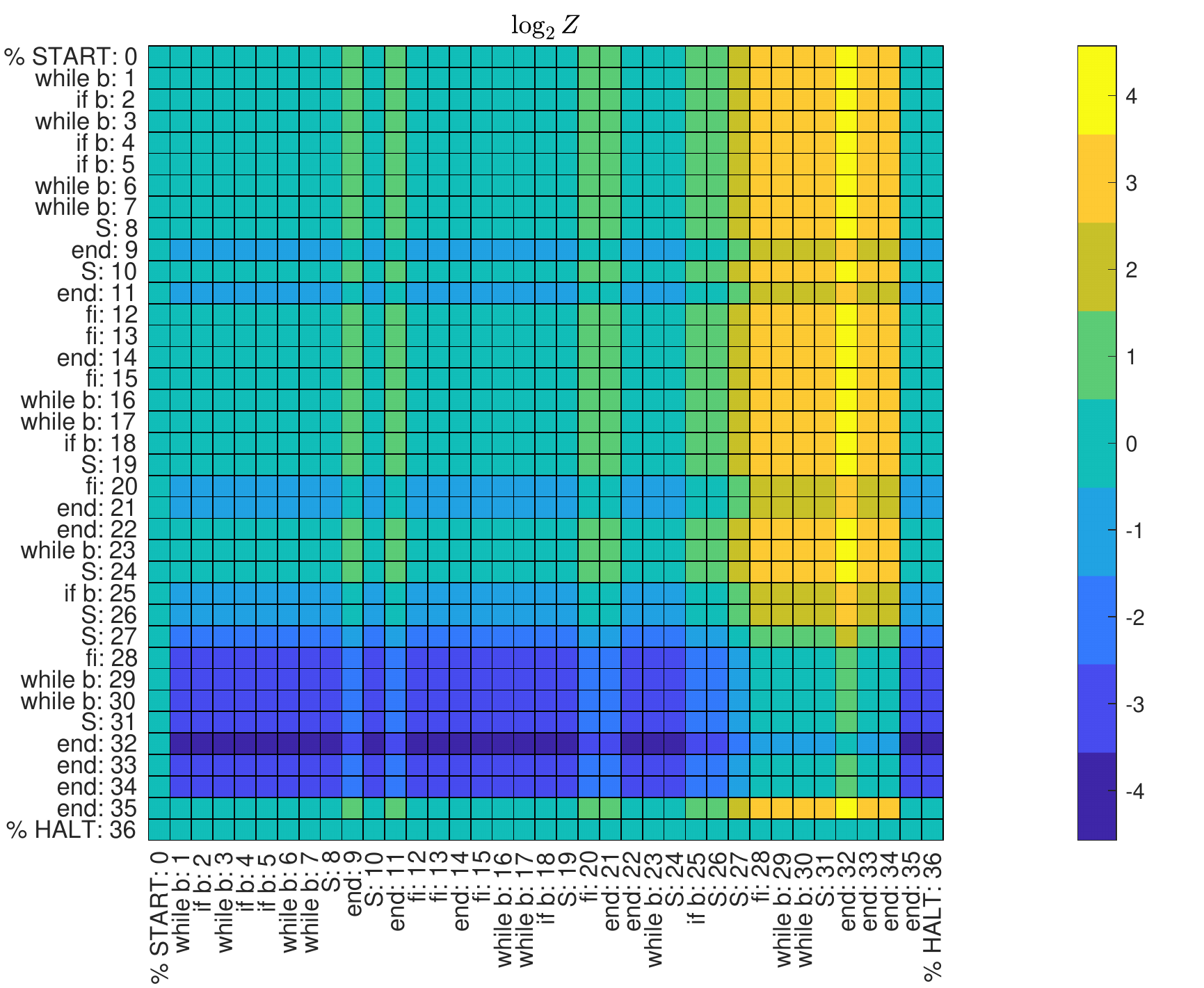} 
  \includegraphics[trim = 20mm 65mm 20mm 70mm, clip, width=.49\textwidth,keepaspectratio]{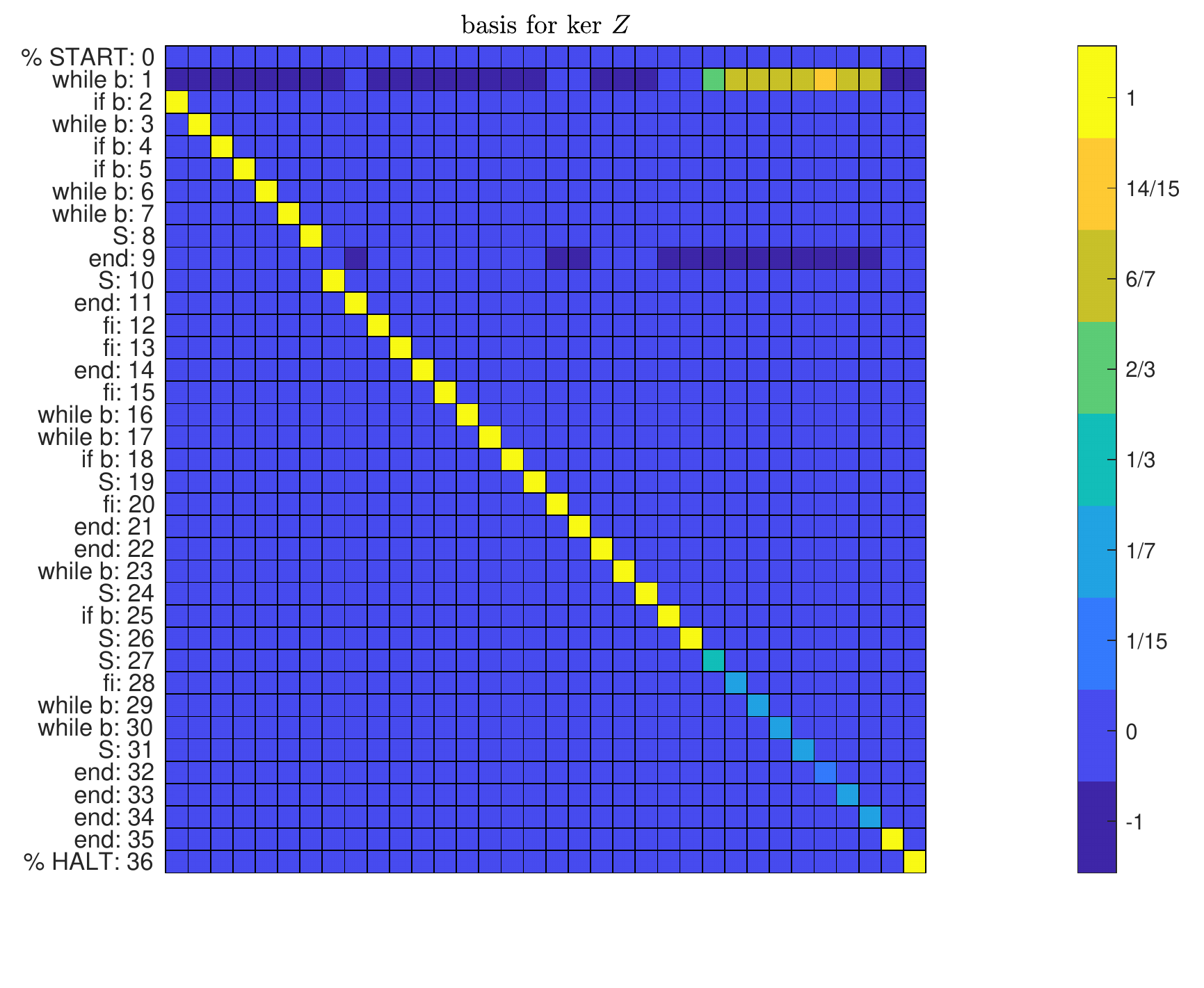} 
  \caption{\label{fig:Matrices} (Left) $\log_2$ of the similarity matrix $Z$ arising from the assignment $\delta_j = 2$ for all $j$ to the program with control flow graph in Figure \ref{fig:ControlFlowGraph}. (Right) A rational basis for the kernel of $Z$, with maximum absolute value of vector entries normalized to unity. Note that the color scheme is quantized/nonlinear.}
\end{figure}

Suppose now that we change one of the $\delta_j$ to equal $8$ instead of $2$. The (maxima of the rows of the) resulting kernels are shown in the left panels of Figure \ref{fig:Kernels}; the right panels are similar but with $\delta_j = 1/4$. Although the magnitude is always unity, the space of weightings encodes globally contextualized information about the program ``geometry.''

\begin{figure}[t]
  \centering
  \includegraphics[trim = 10mm 58mm 10mm 56mm, clip, width=.49\textwidth,keepaspectratio]{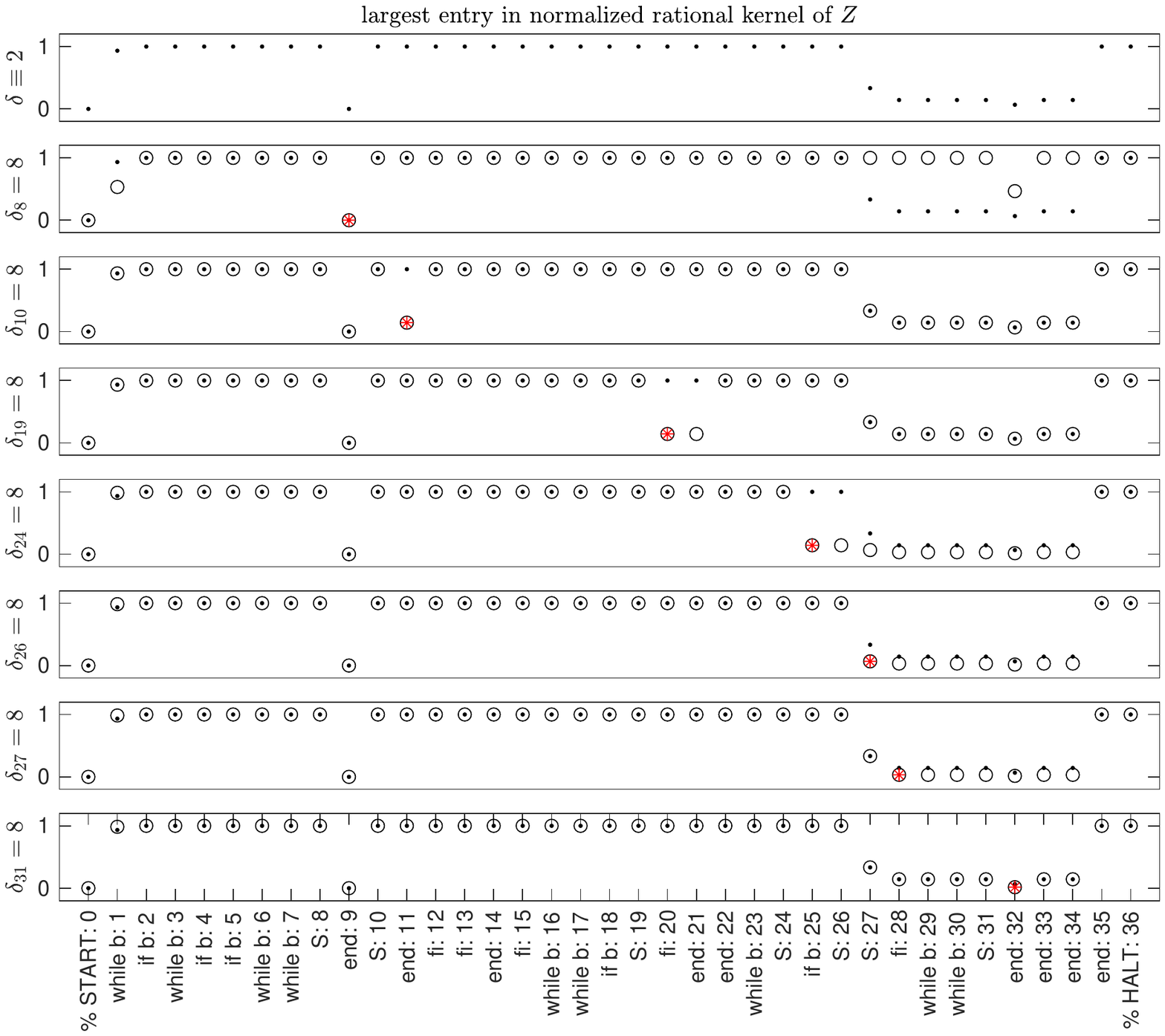} 
  \includegraphics[trim = 10mm 58mm 10mm 56mm, clip, width=.49\textwidth,keepaspectratio]{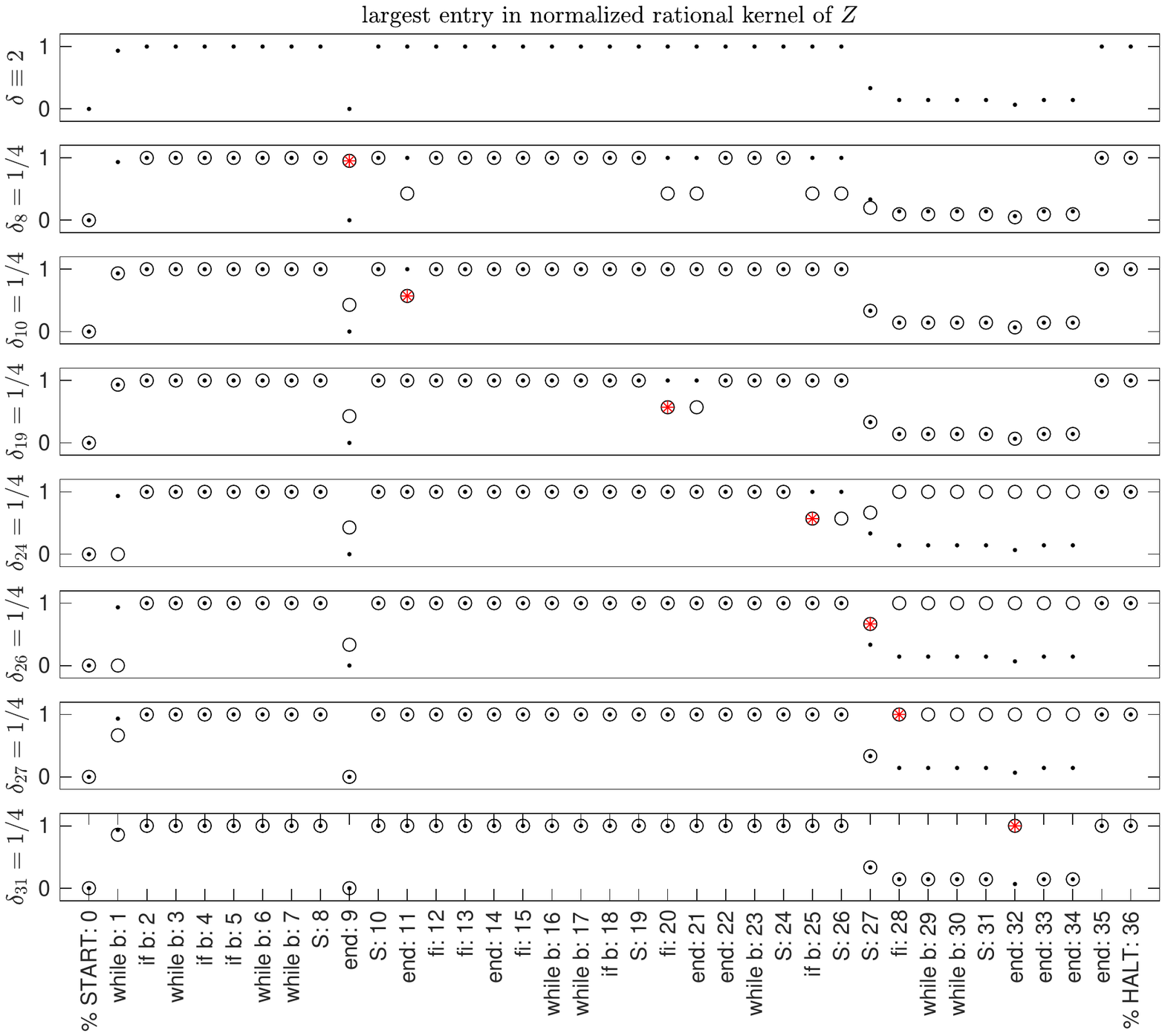} 
  \caption{\label{fig:Kernels} (Top left panel) Maxima of rows of the kernel of $Z$ from figure \ref{fig:Matrices}. (Lower left panels) As in the top panel, but also showing (with $\circ$) maxima for $\delta_j = 8$ with $j = 8, 10, 19, 24, 26, 27, 31$, respectively. The maximum for $j+1$ is shown with a {\color{red}red *}. 
(Right panels) As in the left panels, but taking $\delta_j = 1/4$.}
\end{figure}


%

\end{example}

If the sizes are positive, taking logarithms yields a dissimilarity that further essentially reduces the situation to a case of ordinary Lawvere metric magnitude for arc-weighted DAGs. In fact we can exploit this to gain some intuition for more general digraphs. Suppose all of the sizes on a spanning (poly)tree satisfy $\sigma = \exp(-1)$: then up to degeneracies related to reachability, $Z = \exp[-d]$ where $d$ is the digraph distance. By analogy with Euclidean distance matrices, we thus expect the resulting space of weightings to indicate vertices that are ``large'' by virtue of being ``peripheral'' in a way that sometimes but not always explicitly correlates to degree.

\begin{example}
Let $K^\rightarrow_{n_1,\dots,n_L}$ denote the DAG with $V(K^\rightarrow_{n_1,\dots,n_L}) := \bigsqcup_{\ell = 1}^L K_\ell$ for $K_\ell := [n_\ell]$ and $A(K^\rightarrow_{n_1,\dots,n_L}) := \{(v,v') : v \in K_\ell, v' \in K_{\ell+1}; \ell \in [L-1]\}$. This DAG corresponds to the architecture of a fully connected multilayer perceptron (MLP) with $L$ layers of widths $n_1, \dots, n_L$ \cite{goodfellow2016deep}.

Now consider a sub-DAG $D \subset K^\rightarrow_{n_1,\dots,n_L}$ with arc weights $\sim U([0,1])$ corresponding to a sparsely connected MLP. Fixing $D$ and retaining the arc weights corresponding to a random spanning tree of $U(D)$
\footnote{
We produce a random spanning tree by taking a minimal spanning tree of $U(D)$ augmented with temporary edge weights $\sim U([0,1])$. This spanning tree is generally not \emph{uniformly} random, but such trees can be produced \cite{aldous1990random}. 
} 
defines $Z$ and hence $w$ as a random variable, as shown for $N = 500$ realizations on the right of Figure \ref{fig:MLP} corresponding to the weighted DAG on the left.

\begin{figure}[htbp]
  \centering
  \includegraphics[trim = 30mm 107mm 35mm 110mm, clip, width=.49\textwidth,keepaspectratio]{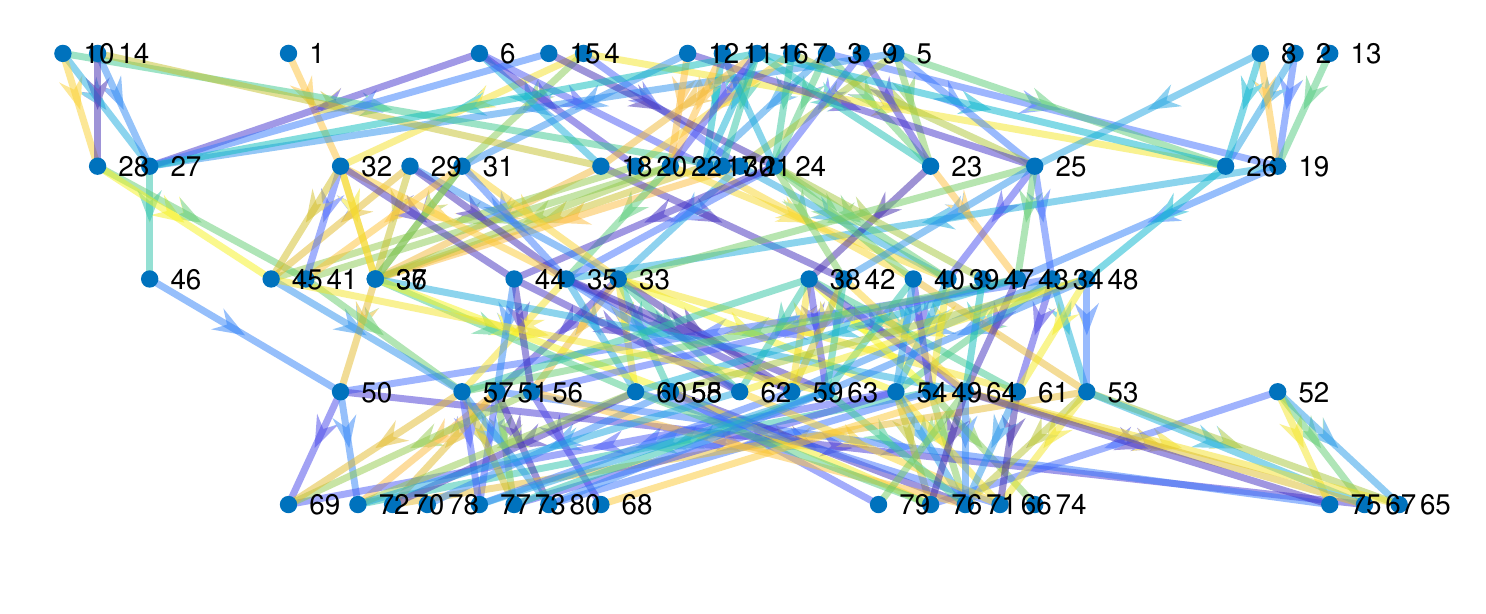} 
  \includegraphics[trim = 20mm 103mm 25mm 102mm, clip, width=.49\textwidth,keepaspectratio]{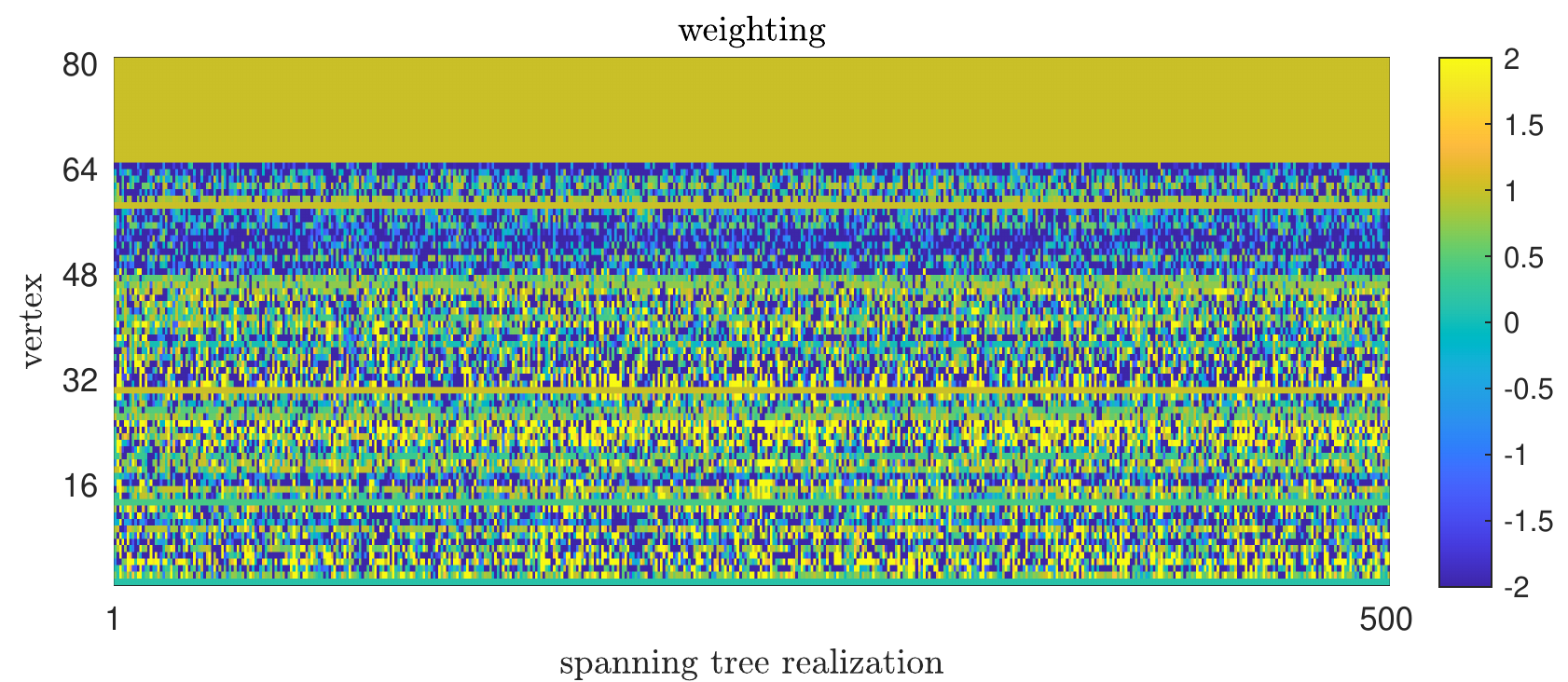} 
  \caption{\label{fig:MLP} (Left) A weighted sub-DAG of $K^\rightarrow_{16,16,16,16,16}$ with arc weights $\sim U([0,1])$ according to the same (rescaled) color scheme on the right. (Right) $N = 500$ realizations of the weighting of the $\mathbb{R}$-category obtained from generating data on a random spanning tree. The color axis has been truncated 
for clarity. The statistical regularity of $w$ is visible as horizontal striping.}
\end{figure}

$w$ is statistically well-behaved: in our experiments, individual components of $w$ all satisfy the hypothesis of being sampled from a normal distribution according to the Anderson-Darling test \cite{razali2011power} with the best significance levels that are provided for in a standard computational implementation.
\footnote{
Multivariate normality tests along the lines of \cite{wang2015matlab} are computationally prohibitive and cannot be nearly as conclusive in our context because of the high dimension.
}
None of the components of $w$ are trivial except those corresponding to vertices with outdegree zero, which have unit values; components for vertices with indegree zero each have fixed values, and components for other vertices are normally distributed. Neighboring vertices tend to have weighting components of opposite signs, consistent with the general intuition in Euclidean space from \cite{willerton2009heuristic,bunch2020practical,huntsman2022diversity} that negative weighting components tend to occur ``just behind a local boundary'' with large positive weighting components. 

In larger networks, presumptive ``near-outlier'' vertices with the statistically least and greatest weighting components tend to be densely connected. Specifically, consider a measure of central tendency $\mathbb{T}$ (e.g., mean, median, etc.) and $t_- < t_+$ such that just a few vertices are in each of the sets $T_- := \{j : \mathbb{T}(w_j) < t_-\}$ and $T_+ := \{j : \mathbb{T}(w_j) > t_+\}$. Now writing $T := T_- \cup T_+$, consider the set $$X := T \cup \left ( \{i : (i,j) \in A(D); j \in T\} \cap \{k : (j',k) \in A(D); j' \in T\} \right ).$$ The sub-DAG induced by $X$ is generally densely connected, though this sub-DAG itself can be noisy. Figure \ref{fig:MLPoutliers} shows an example building on Figure \ref{fig:MLP}; we have observed this behavior across other examples.

\begin{figure}[htbp]
  \centering
  \includegraphics[trim = 30mm 115mm 35mm 112mm, clip, width=.7\textwidth,keepaspectratio]{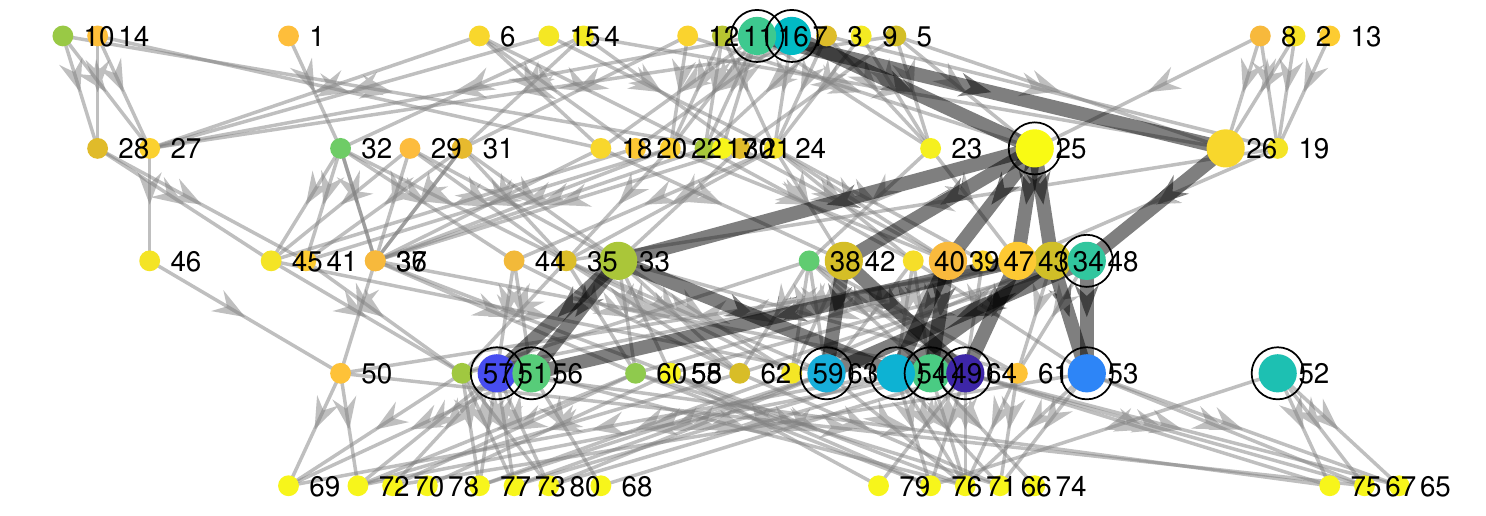} 
  \caption{\label{fig:MLPoutliers} The subgraph of the DAG from Figure \ref{fig:MLP} induced by $X$ for $\mathbb{T} = \text{median}$ and $(t_-,t_+) = (0.14,0.99)$. Vertices are colored by median weighting; $X$ and $T$ are respectively larger and circled.}
\end{figure}

Taken as a whole, these results suggest that enforcing such compositionality of weights in a regularization and/or pruning strategy 
for neural networks might be useful in a way akin to dropout \cite{goodfellow2016deep}. 

The behavior described above (except for many disconnected components in the salient sub-polytree, due to obvious and otherwise irrelevant structure) manifests unambiguously when $D$ is a polytree. Figure \ref{fig:polytree} shows the weighting on a binary polytree with a realization of arc weights $\sim U(\{1,2\})$. 
\footnote{
It turns out that taking arc weights $\sim U([n])$ gives approximately the same result up to affine scaling for any $n \ge 2$ so long as the weights are obtained by quantizing the output of the same pseudorandom number generator.
}

\begin{figure}[htbp]
  \centering
  \includegraphics[trim = 10mm 77mm 5mm 73mm, clip, width=.75\textwidth,keepaspectratio]{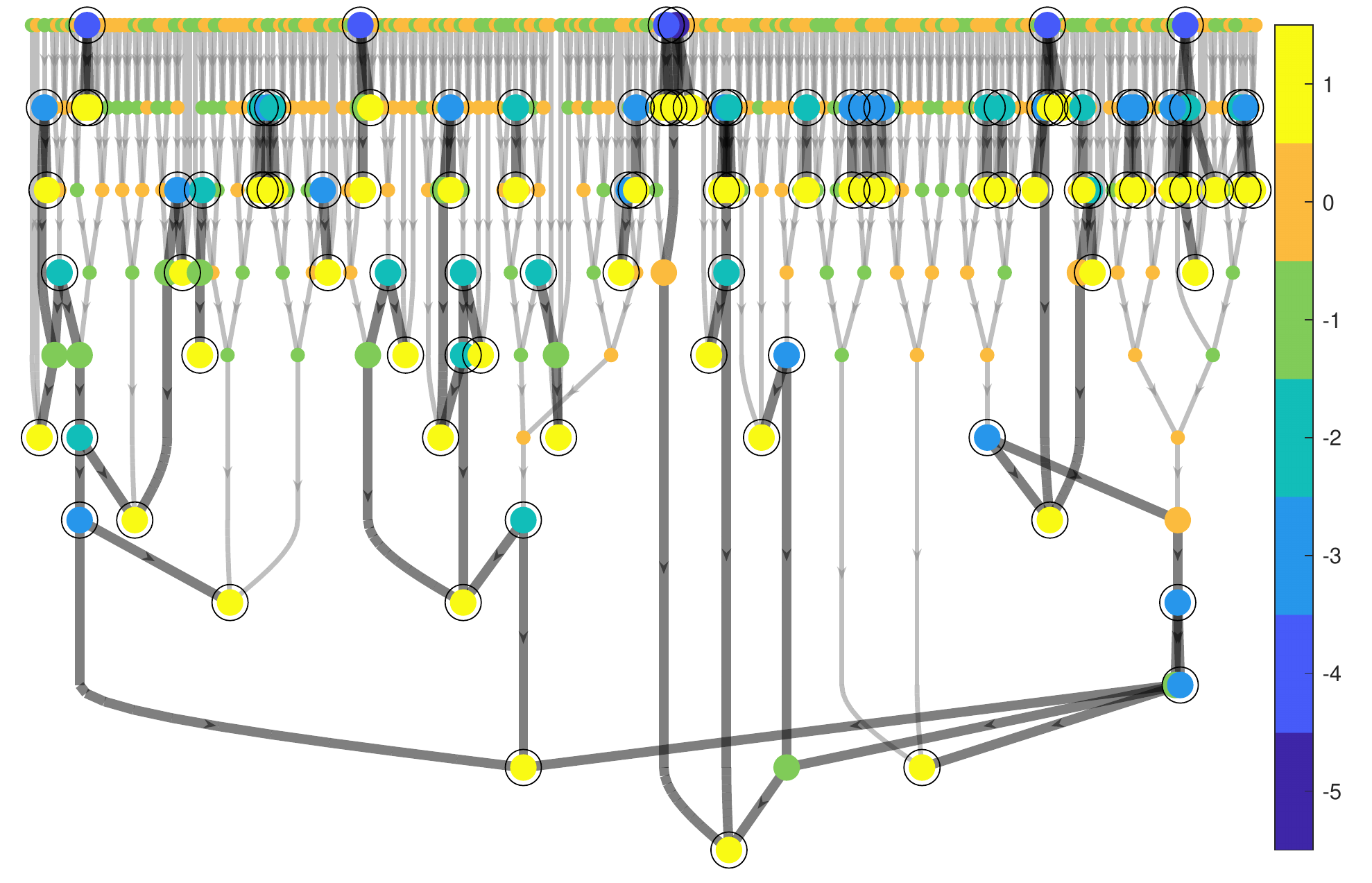} 
  \quad
  \includegraphics[trim = 80mm 70mm 85mm 70mm, clip, width=.18\textwidth,keepaspectratio]{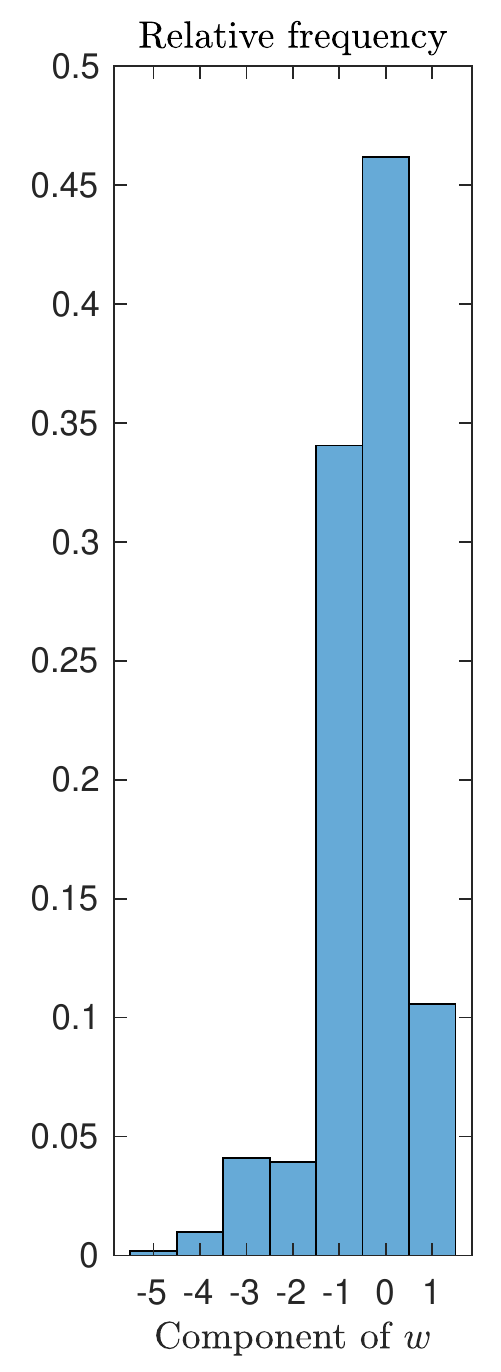} 
  \caption{\label{fig:polytree} (Left) As in Figure \ref{fig:MLPoutliers} for a polytree with arc weights $\sim U(\{1,2\})$ and $\{j: w_j < -1\}$ and $\{j: w_j > 0\}$ as respective analogues of $T_-$ and $T_+$. (Right) Relative frequencies of weighting components.}
\end{figure}

\end{example}

\section{\label{sec:Matrix}Matrix categories}

A \emph{matrix category} is an enriched category whose hom-sets are matrices. In particular, a matrix category is a representation of a digraph \emph{qua} quiver \cite{schiffler2014quiver} that satisfies a compositional coherence condition.
\footnote{
The category of quiver representations has a subcategory of digraph representations, and there is in turn a category of matrix categories of a digraph. The representation theory of these objects is probably interesting in its own right.
}
There are several inequivalent flavors of this construction, though from the perspective of magnitude these factor through the universal construction of \S \ref{sec:Scalar}.

\subsection{\label{sec:Matrix0}Matrix multiplication as monoidal product}

We first consider a model of a finite state machine in which linear maps are associated to the states and/or transitions. This model is applicable to many situations in control theory and/or the analysis of cyber-physical systems, e.g., switched linear systems \cite{liberzon2003switching}.

For $n < \infty$, we can treat the monoid $M_n(R)$ of $n \times n$ matrices over a commutative ring $R$ as a monoidal category with objects $M_n(R)$, only identity morphisms, and with ordinary matrix multiplication as the monoidal product. Suitably reinterpreted (in particular, taking ordered products), the assignment \eqref{eq:treeAssign} defines a $M_n(R)$-category, and any finite $M_n(R)$-category can be realized in this way: similarly, we can write $\mathbf{C}(j,k) = p_j^{-1} p_k$ for suitable $p$. Figure \ref{fig:matrixExample} shows examples for $SL_2(\mathbb{Z})$: these are convenient to write down,  though trivial from the perspective of magnitude.

\begin{figure}[htbp]
  \centering
  \includegraphics[trim = 50mm 90mm 40mm 85mm, clip, width=.49\textwidth,keepaspectratio]{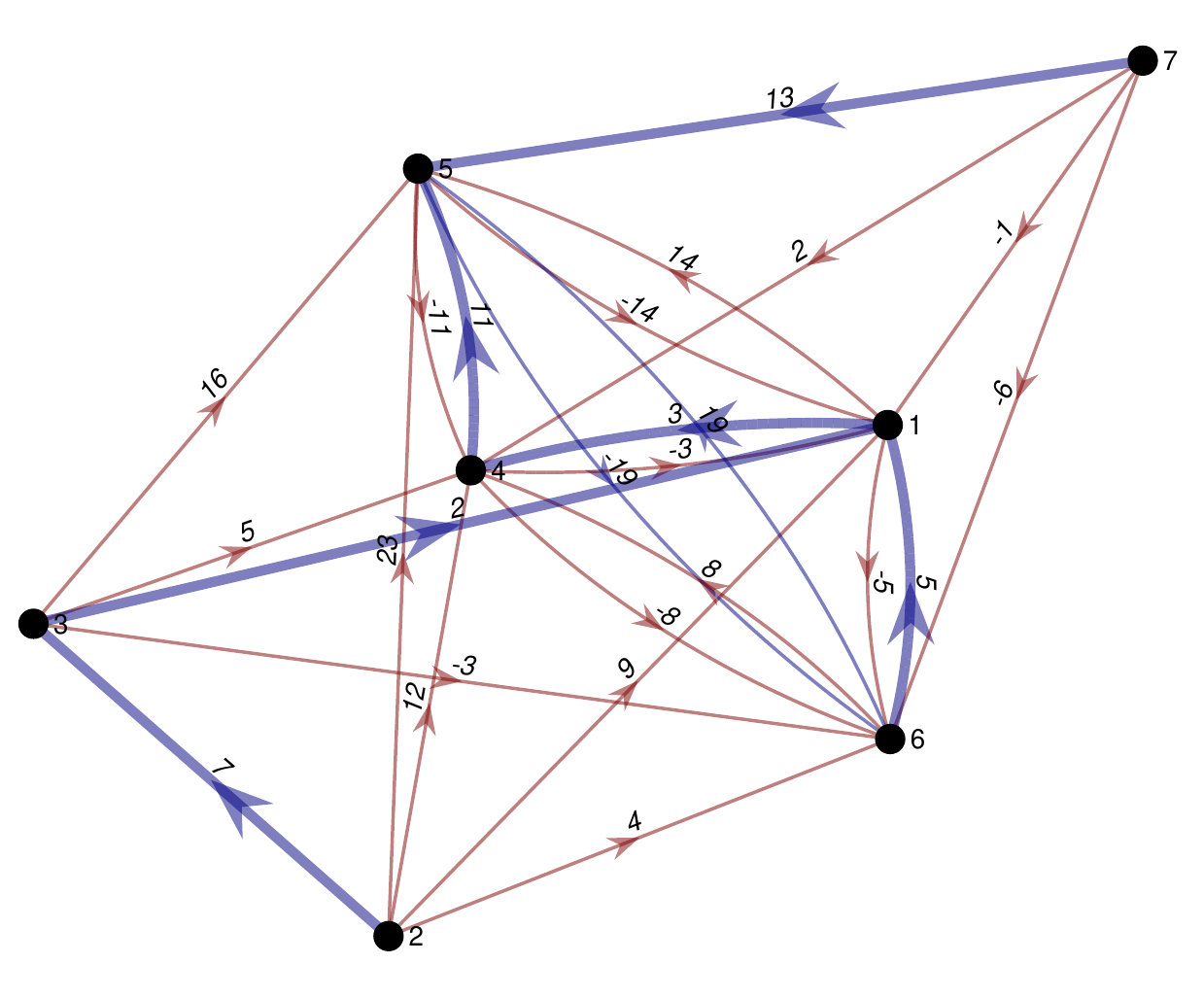} 
  \includegraphics[trim = 50mm 90mm 40mm 85mm, clip, width=.49\textwidth,keepaspectratio]{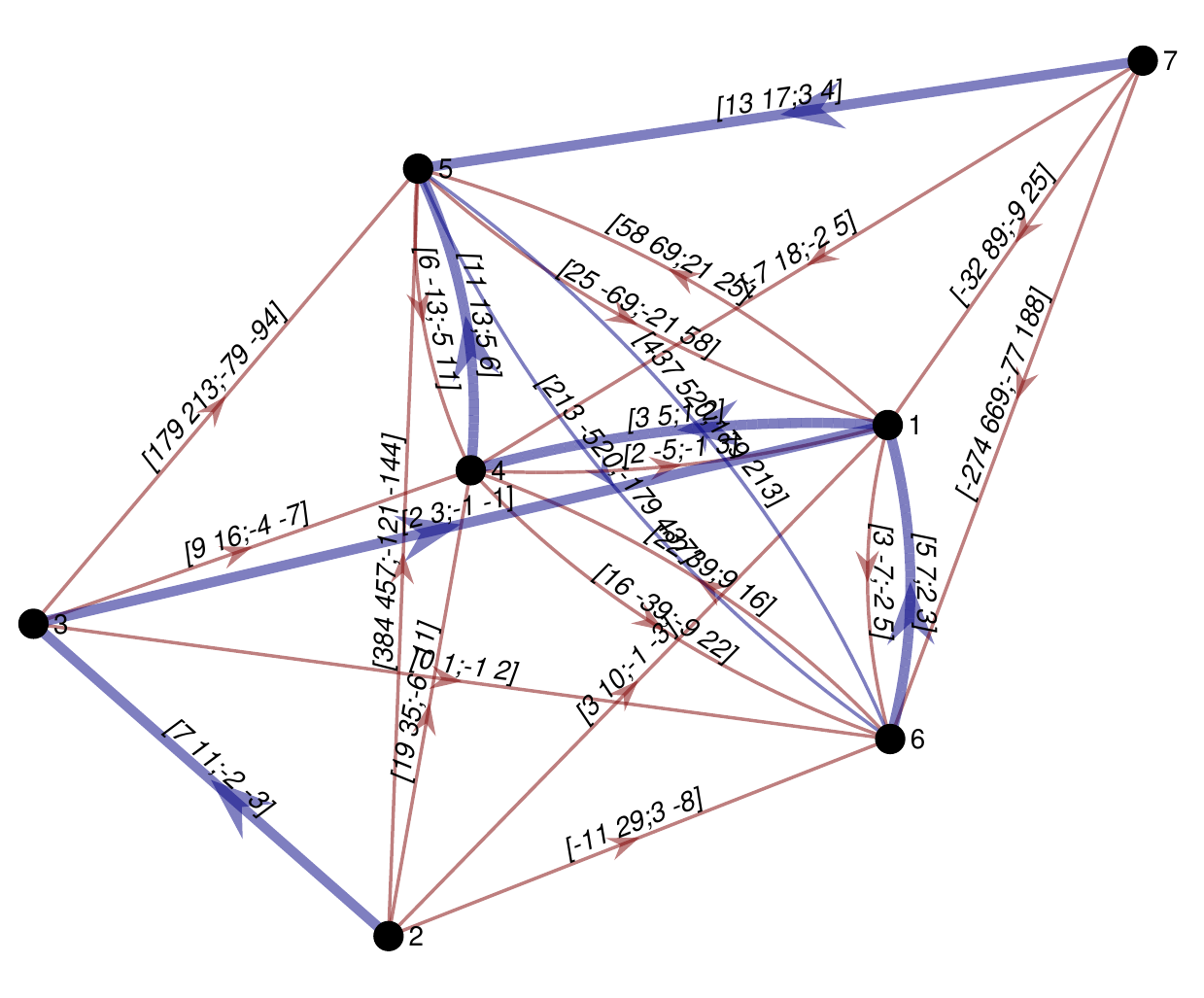} 
  \caption{\label{fig:matrixExample} (Left) As in the right panel of Figure \ref{fig:scalarExample} but labeled with the $(1,2)$ entries of matrices of the form $u(x) := \left ( \begin{smallmatrix}1 & x \\ 0 & 1\end{smallmatrix} \right )$. Note that this essentially uses scalar addition as monoidal product, and the matrices involved here all have unit determinant. Taking $p = u^{\times 7}(9,0,7,12,23,4,10)$ gives $\mathbf{C}(j,k) = p_j^{-1} p_k$.
(Right) A more involved example over $SL_2(\mathbb{Z})$: here taking $p = \left (  \left ( \begin{smallmatrix}3 & 10 \\ -1 & -3\end{smallmatrix} \right ),  \left ( \begin{smallmatrix}1 & 0 \\ 0 & 1\end{smallmatrix} \right ),  \left ( \begin{smallmatrix}7 & 11 \\ -2 & -3\end{smallmatrix} \right ),  \left ( \begin{smallmatrix}19 & 35 \\ -6 & -11\end{smallmatrix} \right ), \left ( \begin{smallmatrix}384 & 457 \\ -121 & -144\end{smallmatrix} \right ), \left ( \begin{smallmatrix}-11 & 29 \\ 3 & -8\end{smallmatrix} \right ), \left ( \begin{smallmatrix}165 & -587 \\ -52 & 185\end{smallmatrix} \right ) \right)$ gives $\mathbf{C}(j,k) = p_j^{-1} p_k$.}
\end{figure}

Meanwhile, the determinant furnishes a suitable size map
and so we can apply the constructions of \S \ref{sec:magnitude}.
\footnote{\label{foot:determinant}
More generally, any well behaved size map factors through the determinant. Note also that determinants similarly inform enrichment over $\text{Arr}(\mathbf{Field})$ or $\text{Arr}(\mathbf{Ring})$ in a way that borrows from the present section and \S \ref{sec:Matrix1}: in these cases, a size map is given by the norm of a field extension or an ideal \cite{janusz1996algebraic}.
}
As far as similarity matrices, weightings, and magnitude are concerned, we can replace matrices with their determinants and apply the results from \S \ref{sec:Scalar} without explicitly forming all of the data for a $M_n(R)$-category. That is, from this perspective we do not really gain anything by considering matrices \emph{versus} scalars: the case of $n > 1$ factors through the case $n = 1$.

\subsection{\label{sec:Matrix1}Kronecker product as monoidal product}

There is also a variant of \S \ref{sec:Matrix0} with the usual (i.e., Kronecker tensor) monoidal product. The most obvious possible application is to quantum circuits \cite{nielsen2010quantum}
\footnote{
Although from the perspective of magnitude our constructions trivialize for unitaries, we can consider the Hermitian matrices obtained via matrix logarithms. This requires considering instead the monoidal product defined by the so-called \emph{Kronecker sum} $H \boxplus H' := H \otimes I_{H'} + I_H \otimes H'$ that satisfies $\exp(H) \otimes \exp(H') = \exp(H \boxplus H')$, where we emphasize here that the matrix exponential $\exp(\cdot)$ is indicated instead of the componentwise exponential $\exp[\cdot]$.
}
although as we shall see considerations of magnitude once again factor through to the scalar case of \S \ref{sec:Scalar} and we postpone an example to \S \ref{sec:stoch}.

Recall that the \emph{arrow category} $\text{Arr}(\mathbf{X})$ of a category $\mathbf{X}$ has as objects the morphisms of $\mathbf{X}$; morphisms given by commutative squares in $\mathbf{X}$; and composition given by concatenation of these commutative squares. If $\mathbf{X}$ is monoidal, then so is $\text{Arr}(\mathbf{X})$.
\footnote{See, e.g., exercise 4 on p. 165 of \cite{mac2013categories} for a more general result.} 
In particular, for $S \in \{\mathbb{R},\mathbb{C}\}$, the objects of $\text{Arr}(\mathbf{FinVect}_S)$ are (linear maps that can be represented as) matrices over $S$, 
\footnote{
Some of our development applies to the category of semimodules over a semiring ($=$ rig), but essential parts require working over $\mathbf{FinVect}_S$ for $S \in \{\mathbb{R},\mathbb{C}\}$. 
}
and $\text{Arr}(\mathbf{FinVect}_S)$ is monoidal with respect to the usual tensor product $\otimes$, i.e., the enriched composition law is
\begin{equation}
\label{eq:enrichedComposition1}
\mathbf{C}(j,k) \otimes \mathbf{C}(k,\ell) = \mathbf{C}(j,\ell).
\end{equation}

We proceed to construct the space of $\text{Arr}(\mathbf{FinVect}_S)$-categories $\mathbf{C}$ with underlying category determined by a finite digraph $D = (V,A)$. As in \S \ref{sec:Matrix0}, the hom-objects of $\mathbf{C}$ are linear maps, but since here the monoidal product is the usual tensor product, the sizes of these matrices vary unless they are all $1 \times 1$. Write $s_{jk} := \dim \text{dom } \mathbf{C}(j,k)$ and $t_{jk} := \dim \text{cod } \mathbf{C}(j,k)$. Then $s$ and $t$ must satisfy \eqref{eq:multiplicative}.

If $(j,k)$ and $(k,\ell)$ are part of a cycle, then so is $(j,\ell)$, which yields a simple proposition.
\begin{proposition} 
\label{proposition:acyclic} 
The system \eqref{eq:multiplicative} does not admit solutions over $\mathbb{Z}_+$ that are nontrivial on arcs that can reach a strong component of $D$. \qed
\end{proposition}

Note that unless $D$ is a DAG, it has nontrivial strong components. Except in degenerate cases, we can perform ``Kronecker division'' of appropriately sized matrices, which leads to the following corollary.

\begin{corollary} 
\label{corollary:multiplicativeSolutions} 
If $D$ is a finite DAG, $P$ is a spanning polytree (i.e., a digraph whose image under $U$ is a tree) of $D$, and $s$, $t$ satisfy \eqref{eq:multiplicative} over $\mathbb{Z}_+$, an $\text{Arr}(\mathbf{FinVect}_S)$-category is specified by a nondegenerate element of $\prod_{(j,j') \in A(P)} M_{(s_{(j,j')},t_{(j,j')})}(S)$, and every $\text{Arr}(\mathbf{FinVect}_S)$-category with underlying category determined by $D$ arises in this way.
\end{corollary}

In the present setting over the base field $S = \mathbb{C}$, the size maps that are also norms are commonly called \emph{cross norms}. The permutation-invariant size maps for matrices over $\mathbb{C}$ are precisely the Schatten $p$-norms $\| \cdot \|_p$, i.e., the $\ell_p$ norms of the vector whose entries are the singular values of a matrix \cite{aubrun2011multiplicative}.
\footnote{
The norms for $p = 1,2,\infty$ respectively are called trace/nuclear; Frobenius/Hilbert-Schmidt; and operator/spectral. It is the case that $\|S\|_\infty \le \|S\|_2 \le \|S\|_1$. The relevant theory also owes much to both Grothendieck \cite{pisier2012grothendieck} and von Neumann \cite{schatten1948cross}.
} 

As in \S \ref{sec:Matrix0}, replacing the arc matrices with sizes factors through to the scalar case of \S \ref{sec:Scalar} all over again.

\subsection{\label{sec:stoch}Stochastic matrices}

The constructions of \S \ref{sec:Matrix0} and \S \ref{sec:Matrix1} trivially specialize to the case where the matrices involved are (row)-stochastic. 
\footnote{
Note however that the inverse of a nonnegative stochastic matrix must contain negative entries, so in either case only DAGs are obvious candidates for applications. The category $\mathbf{FinStoch}$ whose objects and morphisms are respectively finite sets and stochastic maps represented as row-stochastic matrices is discussed in \cite{fritz2020synthetic}.
}
There is not an information-theoretical size map that can take the place of the determinant for an analogue of \S \ref{sec:Matrix0}, as footnote \ref{foot:determinant} points out. However, there is a meaningful information-theoretical size map that produces a nontrivial analogue of \S \ref{sec:Matrix1}, viz. $\mathbf{C}(j,k) \mapsto \exp C(\mathbf{C}(j,k))$, where $C$ indicates the \emph{channel capacity} \cite{shannon1957certain}.  
\footnote{
For background on information theory, see \cite{thomas2006elements}.
B. Fong and D. I. Spivak showed in a personal communication (2019) that channel capacity is a strict monoidal lax functor $(\mathbf{FinStoch},\otimes) \rightarrow ([0,\infty],(\infty,\min),\ge,(+,0))$, where we indicate the lax monoidal po-monoid with one object $*$, with $\text{Hom}(*,*) = [0,\infty]$, with identity $\infty$ and composition given by $\min$, with local partial order (po) structure given by the usual $\ge$, and with monoidal structure given by $+$ and $0$. This result also holds when replacing both i) the Kronecker/tensor $\otimes$ monoidal structure on $\mathbf{FinStoch}$ with the direct sum $\oplus$ monoidal structure and ii) channel capacity $C$ with its exponent $\exp C$.
}

\begin{example}

A \emph{discrete memoryless channel} (DMC) is specified by a matrix of conditional probabilities $W_{jk} = \mathbb{P}(y_k|x_j)$, where $x_j$ and $y_k$ respectively denote input and output symbols corresponding to realizations of random variables $X$ and $Y$. $W_{jk}$ is the probability that if Alice transmits $x_j$, then Bob receives $y_k$. An input distribution $p_j = \mathbb{P}(x_j)$ yields joint probabilities $V_{jk} := \mathbb{P}(x_j,y_k) = p_j W_{jk}$. The channel coding theorem states that reliable communication is possible iff the ratio of informative to transmitted bits is less than the channel capacity $C := \sup_p I(X;Y)$, where $I(X;Y)$ is the mutual information.

Let $D$ be given with arc data as in Figure \ref{fig:Channels}, and where $W^{(i)}$ indicates the matrix for a DMC with corresponding capacity $\log c_i$. The resulting similarity matrix is 
$$
Z = \begin{pmatrix}
1 & 0 & c_1 c_3 & c_1 & c_1 c_3 c_4 & c_1 c_3 c_5 \\
0 & 1 & c_2 c_3 & c_2 & c_2 c_3 c_4 & c_2 c_3 c_5 \\
0 & 0 & 1 & 0 & c_4 & c_5 \\
0 & 0 & 0 & 1 & c_3 c_4 & c_3 c_5 \\
0 & 0 & 0 & 0 & 1 & 0 \\
0 & 0 & 0 & 0 & 0 & 1
\end{pmatrix}
$$
and the equation $Zw = 1$ can be solved by hand to yield
$$w = (1-[1-c_4]c_1c_3-[1-c_3c_5]c_1, 1-[1-c_5]c_2c_3-[1-c_3c_4]c_2, 1-c_4-c_5, 1-c_3c_4-c_3c_5, 1, 1)^T.$$ There is an asymmetry in form between $w_3 = 1-c_4-c_5$ and $w_4 = 1-c_3c_4-c_3c_5$. This reflects the fact that even though the end-to-end capacity is path-independent, the capacity of channels to node 3 differs from the capacity of channels to node 4. Note that if we switch the channels from vertex 2 with each other and the channels to vertex 6 with each other that this asymmetry disappears.
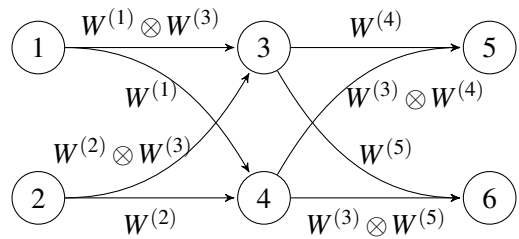
\begin{figure}[htbp]
	\begin{center}
	\begin{tikzpicture}[scale=1,->,>=stealth',shorten >=1pt,every node/.style={transform shape}]
		\node [draw,circle,minimum size=7mm] (v1) at (0,2) {1};
		\node [draw,circle,minimum size=7mm] (v2) at (0,0) {2};
		\node [draw,circle,minimum size=7mm] (v3) at (3,2) {3};
		\node [draw,circle,minimum size=7mm] (v4) at (3,0) {4};
		\node [draw,circle,minimum size=7mm] (v5) at (6,2) {5};
		\node [draw,circle,minimum size=7mm] (v6) at (6,0) {6};
		\path [->,color=black] (v1) edge node [above] {$W^{(1)} \otimes W^{(3)}$} (v3);
		\path [->,color=black,out=0,in=120,looseness=1] (v1) edge node [pos = 0.4, below] {$W^{(1)}$ \quad } (v4);
		\path [->,color=black,out=0,in=-120,looseness=1] (v2) edge node [pos = 0.4, above] {$W^{(2)} \otimes W^{(3)}$ \quad \quad \quad } (v3);
		\path [->,color=black] (v2) edge node [below] {$W^{(2)}$} (v4);
		\path [->,color=black] (v3) edge node [above] {$W^{(4)}$} (v5);
		\path [->,color=black,out=-60,in=-180,looseness=1] (v3) edge node [pos = 0.6, above] { \quad $W^{(5)}$} (v6);
		\path [->,color=black,out=60,in=-180,looseness=1] (v4) edge node [pos = 0.6, below] { \quad \quad \quad $W^{(3)} \otimes W^{(4)}$} (v5);
		\path [->,color=black] (v4) edge node [below] {$W^{(3)} \otimes W^{(5)}$} (v6);
	\end{tikzpicture}
	\end{center}
	\caption{\label{fig:Channels} A $\mathbf{FinStoch}$-category (with nontrivial compositions suppressed for clarity).}
\end{figure}
\end{example}

\subsubsection{\label{sec:capacity}Sidebar on channel capacity and coweightings}

In the spirit of understanding and applying magnitude in atypical contexts, it can be instructive to consider linear equations $Zw = 1$ and/or $vZ = 1^T$ where the dissimilarity matrix $Z$ cannot be written in the form $\exp[-td]$ for any Lawvere metric $d$. For instance, the capacity-achieving input distribution of a DMC with invertible channel matrix $W$ can be thought of as such a normalized coweighting. 
The key to realizing this is a classical formula due to Muroga \cite{muroga1953capacity}. 

Let $W$ be an invertible channel matrix with $M := W^{-1}$ and define $H_j := -\sum_k W_{jk} \log W_{jk}$. Furthermore, define $v_j := \sum_i M_{ij} \exp \left ( - \sum_k M_{ik} H_k \right )$. The Muroga formula states that if $v > 0$, then $e^C = \sum_j \exp \left ( - \sum_k M_{jk} H_k \right )$, and $p := e^{-C} v$ is a capacity-achieving distribution. Now $v = \exp(-MH)M = 1^T \Delta(\exp(-MH))M$, so writing $Z := W \Delta(\exp(MH))$, we obtain $vZ = 1^T$. However, the putative distance $-t^{-1}[\log W_{jk} - (MH)_k]$ is generically negative on the diagonal, though also approximately equal to a constant times $1-\delta_{jk}$ for a ``good channel'' $W \approx I$. The better the DMC is, the closer its induced ``metric geometry'' (where the negative diagonal entails the quotation marks) is to that of a regular simplex.

\section{\label{sec:Remarks}Remarks on ``flow-like'' graphs augmented with data}

An interesting if still speculative possibility for deeper applications to processes, computer programs, etc. is to consider categories of the form $\mathbf{F} \times \mathbf{M}$ where $\mathbf{F}$ is a suitable category of ``flow-like'' digraphs that have single inputs and outputs and $\mathbf{M}$ is a suitable category of data (perhaps scalar if not in the spirit of a nondegenerate matrix category). The monoidal structure on $\mathbf{F}$ is simply gluing the output of one `flow-like'' digraph to the input of another along the lines of \cite{huntsman2022magnitude}. That is, each ``flow-like'' graph has associated data that compose nicely under series composition. A specific instantiation of $\mathbf{F}$ that is algorithmically and categorically well-behaved (in particular, semicartesian) is likely obtainable via minor technical modifications to the notion of a \emph{two-terminal graph} \cite{vanhatalo2008refined}.

In order to have a useful notion of magnitude in this context, it is necessary to have a size function (or for magnitude homology \cite{leinster2021magnitude}, a strong symmetric monoidal functor from a semicartesian symmetric monoidal category to a suitable abelian symmetric monoidal category) that consolidates the graphical and matrix data. A good candidate for constructing something in this vein appears to be the zeta function of a finite Markov chain \cite{parry1977block} that encodes (e.g.) graph traversal probabilities.

\section*{Acknowledgement}

This research was developed with funding from the Defense Advanced Research Projects Agency (DARPA). The views, opinions and/or findings expressed are those of the author and should not be interpreted as representing the official views or policies of the Department of Defense or the U.S. Government. Distribution Statement ``A'' (Approved for Public Release, Distribution Unlimited).

\bibliographystyle{eptcs}
\bibliography{linear} 


\end{document}